\documentclass[final,3p,times]{elsarticle}

\usepackage{lineno,hyperref}
\modulolinenumbers[5]
\usepackage{amssymb}
\usepackage{amsmath}
\usepackage{amsthm}
\usepackage{comment}

\journal{Journal of \LaTeX\ Templates}

\theoremstyle{plain}
 \newtheorem{thm}{Theorem}[section]
 
 \newtheorem{lem}{Lemma}[section]
 
\theoremstyle{definition}

\theoremstyle{remark}
 
  \newtheorem*{ack}{Acknowledgement}
 \numberwithin{equation}{section}
 
\makeatletter
\def\ps@pprintTitle{%
  \let\@oddhead\@empty
  \let\@evenhead\@empty
  \let\@oddfoot\@empty
  \let\@evenfoot\@oddfoot
}
\makeatother

\begin{document}

\begin{frontmatter}

\title{Global existence of solutions to Keller-Segel chemotaxis system with heterogeneous logistic source and nonlinear secretion}

\author[GA]{Gurusamy Arumugam}
\address[GA]{Department of Mathematics, National Institute of Technology Calicut,  Kerala, India.}
\ead{guru.poy@gmail.com}

\author[AD]{Asha K. Dond}
\address[AD]{School of Mathematics, Indian Institute of Science Education and Research, Thiruvananthapuram, Kerala, India.}
\ead{ashadond@iisertvm.ac.in}

\author[AE]{Andr\'e H. Erhardt}
\address[AE]{Weierstrass Institute for Applied Analysis and Stochastics, Mohrenstraße 39, 10117 Berlin, Germany}
\ead{andre.erhardt@wias-berlin.de}

\begin{abstract}
We study the following Keller-Segel chemotaxis system with logistic source and nonlinear secretion:
  \begin{align*}
  	u_t&=\Delta u- \nabla\cdot(u\nabla v)+\kappa(|x|)u-\mu(|x|)u^p\quad\text{and}\quad
  	0=\Delta  v-v+u^\gamma,
\end{align*}
where 
$\kappa(\cdot),~\mu(\cdot):[0,R]\rightarrow [0,\infty)$, $\gamma\in (1,\infty)$, $p\in(\gamma+1,\infty)$ and   $\Omega \subset \mathbb{R}^n, n\geq 2$. For this system,  we prove the global existence of  solutions under suitable assumptions on the initial condition and the functions $\kappa(\cdot)$ and $\mu(\cdot).$ 
\end{abstract}


\end{frontmatter}

\vspace{-0.28cm}
\section{Introduction} 
In this paper, we investigate the following Keller-Segel chemotaxis system, which depends on a logistic source term~\cite{logisticsources,logisticsources1,logisticsources2,logisticsources3} and nonlinear secretion~\cite{secretion2,secretion3,secretion4}, and reads as follows:
\begin{align}
\begin{cases}
 u_t&=\Delta u- \nabla\cdot(u\nabla v)+g(x,u)
 \\
0&=\Delta  v-v+u^\gamma
\end{cases}\quad
\mbox{in}\ \ \Omega\times (0,T),\label{1}
\end{align}
with initial values $u(\cdot,0)=u_0(x)$ for $x\in\Omega$ and homogeneous Neumann boundary conditions $\frac{\partial u}{\partial \nu}=\frac{\partial v}{\partial \nu}=0$ on $\partial\Omega\times (0,T)$, where $\Omega\subset\mathbb{R}^n, n\geq 2$, $u$ denotes the cell density and $v$ denotes the concentration of chemical signal. In addition, we consider the environment depending logistic source $g(x,u)=\kappa(|x|)u-\mu(|x|)u^p$. The function $\kappa(|x|)$ represents 
the self-growth, while $\mu(|x|)$ the self-limitation of the mobile species. This type of system is relevant in the modelling of micro- and macroscopic population dynamics  or tumour invasion processes, see e.g.~\cite{logisticsources1}.

Historically, mathematical modelling of chemotaxis phenomenon dates to the pioneering works of Patlak in the 1950s \cite{pat1953} and Keller-Segel in the 1970s \cite{ks1971,kss1971}. The study of Keller-Segel (type) system is motivated by numerous applications, see for instance 
\cite{ph2001,MR4188348}.  The general form of the Keller-Segel chemotaxis system is given by  
\begin{align}
\hfill u_t&=\nabla\cdot(\phi(u,v)\nabla u- \psi(u,v)\nabla v) + f(u,v)\quad\text{and}\quad
\tau v_t=d\Delta v+ g(u,v)u -h(u,v) v,
\label{KS_general}
\end{align}
where $u$ represents the cell (or organism) density on a given domain $\Omega\subset\mathbb{R}^n$ and $v$ denotes the concentration of the chemical signal.  The cell dynamics derive from  population kinetics and movement, the latter comprising a diffusive flux modelling undirected (random) cell migration and an advective flux with velocity dependency on the gradient of the signal. The motility function $\phi(u,v)$ describes the diffusivity of the cells and $\psi(u,v)$  represents the chemotactic sensitivity. 
The function $f(u,v)$ describes cell growth and death, while the functions $g(u,v)$ and $h(u,v)$  are kinetic functions that describe production and degradation of the chemical signal, respectively. Organisms or cell moves from a lower concentration to its higher concentration of the chemoattractant, which is known as positive chemotaxis. 

Notice that, the general form of chemotactic term is of the form $\psi(u,v,|\nabla v|)$~\cite{nt2018}. The important properties of system \eqref{KS_general} are self aggregation phenomenon and spatial pattern formation. Due to the important applications of chemotaxis in medical and biological sciences,  the  research on chemotaxis phenomenon has become an increasing interest in applied mathematics.

\paragraph*{Main results} The principal purpose of this work is to provide the global existence of solutions to the parabolic-elliptic Keller-Segel system \eqref{1}. The main result reads as follows:
\begin{thm}\label{main}
Let $\Omega$ be a bounded and smooth domain in $\mathbb{R}^n,n\geq 2.$ Let $\mu_1>0,~\gamma>1,~p>\gamma+1$, $q>\frac{n\gamma}{2}$ and $\kappa,~\mu \in C^0([0,R])\cap C^1((0,R))$ and let $g(x,u)=\kappa(|x|)u-\mu(|x|)u^p$. 
In addition, we assume that $0<\alpha<2\frac{p-1-\gamma}{q+\gamma}$. If 
\begin{align}
\mu(s)\geq \mu_1 s^\alpha, \ \ \mbox{for all}\ \ s\in [0,R]\label{r}
\end{align}
then there exists  a global-in-time classical solution to system \eqref{1} for any nonnegative initial datum $u_0\in C^0(\overline{\Omega}).$ 
\end{thm}
Notice that the assumption on $\gamma$, i.e. $\gamma>1$, is required to guarantee the existence of a unique solution. Furthermore, we would like to refer to \cite[Proposition 1.4]{logisticsources1} for the case $n=2$ and $\gamma=1$.
\paragraph*{Plan of the paper} In Section 2 we state the local existence of classical solutions and certain preliminaries. Then, in Section 3 we derive a $L^q-$bound for the cell density $u$. In Section 4, we derive the $L^\infty-$bound for $u$ and finally, we provide the global existence of solution in Section 5. 
\section{Preliminaries}
In this section, we state the local existence of solutions to system \eqref{1}, which is based on a fixed point argument. 
\begin{lem}\label{localexistence} Under the assumption of Theorem \ref{main},
there exists $T_{max}\in (0,\infty]$ and a local-in-time classical solution $(u,v)$ to system \eqref{1} uniquely determined such that 
\begin{align}
u\in C^0(\overline{\Omega}\times[0,T_{max}))\cap C^{2,1}(\overline{\Omega}\times(0,T_{max}))\\\
v\in \cap_{q>n} C^0([0,T_{max});W^{1,q}(\Omega))\cap C^{2,0}(\overline{\Omega}\times(0,T_{max}))
\end{align}
and $\int_\Omega v(\cdot,t)\mathrm{d}x=0$ for all $t\in(0,T_{max}).$
Moreover, this solution is nonnegative in $u$, radially symmetric if $u_0$ is radially symmetric and such that if $T_{max}<\infty$, then 
\begin{align}
\lim_{t\rightarrow T_{max}}\sup \|u(\cdot,t)\|_{L^\infty(\Omega)}=\infty.\label{bc}
\end{align}
\end{lem}
We omit the proof of Lemma \ref{localexistence}, since it is similar to the one in \cite[Theorem 2.1]{wt2007}.
In addition, the solution $u$ satisfies the following $L^1$ estimate:
\begin{lem}\label{l22} Under the assumption of Lemma \ref{localexistence} and for any nonnegative function $u_0\in C^0(\overline{\Omega})$, let $(u,v)$ be the local in time classical solution of system \eqref{1}. Then, $u$ satisfy 
\begin{align*}
\int_\Omega u(\cdot,t)\mathrm{d}x\leq K, \ \mbox{for all}\ \  t\in (0,T_{max}),\ \text{where} \ K=\max\left\{\int_\Omega u_0, \frac{k_1}{\mu_1}|\Omega|\right\}.
\end{align*} 
\end{lem}
The proof of Lemma \ref{l22} is similar to the proof of \cite[Lemma 2.2]{wang2020}. Furthermore, to make the paper self-contained we provide some basic inequalities.  
\begin{lem}[Interpolation inequality, \cite{ev}, page 623]Assume $1\leq s\leq r\leq t\leq \infty$ and 
$\frac{1}{r}=\frac{\theta}{s}+\frac{(1-\theta)}{t}.$
Suppose also $u\in L^s(\Omega)\cap L^t(\Omega).$ Then $u\in L^r(\Omega)$ and 
\begin{align}
\|u\|_{L^r(\Omega)}\leq \|u\|_{L^s(\Omega)}^\theta\|u\|_{L^t(\Omega)}^{1-\theta}.\label{ip}
\end{align}
\end{lem}
\begin{lem}[Lemma 5.1 \cite{tem}]\label{lemmaODE}
Let $y(t)\geq 0$ satisfy 
\begin{align*}
\begin{cases}
\displaystyle{\frac{\mathrm{d}y}{\mathrm{d}t}+c_1y^\alpha\leq c_2,}\\
y(0)=y_0
\end{cases}
\end{align*}
with some constants $c_1,c_2>0$ and $\theta\geq 1.$ Then we have 
\begin{align*}
y(t)\leq \max\bigg\{y_0,\bigg(\frac{c_2}{c_1}\bigg)^{\frac{1}{\alpha}}\bigg\}, \ t>0.
\end{align*}
\end{lem}
In addition, we will need the following lemma, which is established in \cite[Lemma 1.3]{mw12010} and reads as follows:
\begin{lem}\label{11}
Let $(e^{t\Delta})_{t\geq 0}$ be the Neumann heat semigroup in $\Omega$, and $\lambda_1>0$ denote the first nonzero eigenvalue of $-\Delta$ in $\Omega$ under the Neumann boundary conditions. Then there exist constants  $C_1,\ldots, C_4$ depending on $\Omega$ only which have the following properties
\begin{itemize}
\item[(i)] If $1\leq q\leq p\leq\infty$ then
\begin{align*}
\|e^{t\Delta}w\|_{L^p(\Omega)}\leq C_1\bigg(1+t^{-\frac{n}{2}(\frac{1}{q}-\frac{1}{p}})\bigg)e^{-\lambda_1t}\|w\|_{L^q(\Omega)} \ \ \mbox{for all}\ t>0
\end{align*}
holds for all $w\in L^q(\Omega)$ satisfying $\int_\Omega w=0.$
\item[(ii)] If $1\leq q\leq p\leq \infty$ then 
\begin{align*}
\|\nabla e^{t\Delta}w\|_{L^p(\Omega)}\leq C_2\bigg(1+t^{-\frac{1}{2}-\frac{n}{2}(\frac{1}{q}-\frac{1}{p}})\bigg)e^{-\lambda_1t}\|w\|_{L^q(\Omega)} \ \ \mbox{for all}\ \ t>0
\end{align*}
is true for each $w\in L^q(\Omega).$
\item[(iii)] If $2\leq p<\infty$ then $$\|\nabla e^{t\Delta}w\|_{L^p(\Omega)}\leq C_3e^{-\lambda_1t}\|w\|_{L^p(\Omega)}$$ for all $t>0$ 
is valid for all $w\in W^{1,p}(\Omega).$
\end{itemize}
\item[(iv)] Let $1<q\leq p<\infty$. Then 
\begin{align*}
\|e^{t\Delta}\nabla \cdot w\|_{L^p(\Omega)}\leq C_4\bigg(1+t^{-\frac{1}{2}-\frac{n}{2}(\frac{1}{q}-\frac{1}{p}})\bigg)e^{-\lambda_1t}\| w\|_{L^q(\Omega)} \ \ \mbox{for all}\ \ t>0
\end{align*}
holds for all $w\in(C_0^\infty(\Omega))^n.$ Consequently, for all $t>0$ the operator $e^{t\Delta}\nabla \cdot $ possesses a uniquely determined extension to an operator from $L^q(\Omega)$ into $L^p(\Omega)$, with norm controlled according to the last inequality. 
\end{lem}
\section{$L^q-$bound  for $u$}
In this section, we establish the $L^q-$bound for the solution $u$ of the system~\eqref{1}.
\begin{lem}\label{l11}
Assume that $\gamma>1$, $p>\gamma+1$, $q>\frac{n\gamma}{2}$ and let 
$0<\alpha<2\frac{p-1-\gamma}{q+\gamma}$. Then
\begin{align*}
\|u(\cdot,t)\|_{L^q(\Omega)}\leq c \ \ \mbox{for all}\ \ t\in (0,T_{max}).
\end{align*}
\end{lem}
\begin{proof}
Testing the first equation of the system \eqref{1} with $u^{q-1}, q>\frac{n\gamma}{2}>1$, we get
\begin{align*}
\int_\Omega u_t u^{q-1}\mathrm{d}x=\int_\Omega \Delta u u^{q-1}\mathrm{d}x-&\int_\Omega \nabla\cdot(u\nabla v)\cdot u^{q-1}+\int_\Omega \kappa(|x|)u u^{q-1}\mathrm{d}x-\int_\Omega \mu(|x|) u^p u^{q-1}\mathrm{d}x.
\end{align*} 
Integration by parts leads to 
\begin{align}
\begin{split}
\frac{1}{q}\frac{\mathrm{d}}{\mathrm{d}t}\int_\Omega u^q\mathrm{d}x=&-\frac{4(q-1)}{q^2}\int_\Omega |\nabla u^{\frac{q}{2}} |^2\mathrm{d}x+(q-1)\int_\Omega u^{q-1}\nabla v\cdot \nabla u\mathrm{d}x+\int_\Omega \kappa(|x|)u^q\mathrm{d}x-\int_\Omega \mu(|x|) u^{p+q-1}\mathrm{d}x.
\end{split}
\label{eq2}
\end{align} 
Next, multiplying the second equation of the system \eqref{1} by $u^{q}$, we get 
\begin{align*}
\int_\Omega \Delta v u^{q}\mathrm{d}x-\int_\Omega u^{q} v\mathrm{d}x+\int_\Omega u^{q+\gamma}\mathrm{d}x=0
\end{align*}
and integration by parts yields 
\begin{align*}
-q\int_\Omega  u^{q-1}\nabla u\nabla v\mathrm{d}x-\int_\Omega u^{q} v\mathrm{d}x+\int_\Omega u^{q+\gamma}\mathrm{d}x=0.
\end{align*}
This implies
\begin{align}
q\int_\Omega  u^{q-1}\nabla u\nabla v\mathrm{d}x=-\int_\Omega u^{q} v\mathrm{d}x+\int_\Omega u^{q+\gamma}\mathrm{d}x,\label{eq3}
\end{align}
Substituting \eqref{eq3} into \eqref{eq2}, we obtain
\begin{align}
\begin{split}
\frac{1}{q}\frac{\mathrm{d}}{\mathrm{d}t}\int_\Omega u^q\mathrm{d}x+\frac{4(q-1)}{q^2}\int_\Omega |\nabla u^{\frac{q}{2}} |^2\mathrm{d}x&\leq \frac{(q-1)}{q}\int_\Omega u^{q+\gamma} \mathrm{d}x+\int_\Omega \kappa(|x|)u^q\mathrm{d}x-\int_\Omega \mu(|x|) u^{p+q-1}\mathrm{d}x.
\end{split}\label{eq4}
\end{align} 
Estimating the integrand of the first term on the right-hand side of \eqref{eq4}, we get 
  \begin{align}
 \begin{split}
u^{q+\gamma} 
&\leq \epsilon_1 |x|^\alpha(u^{q+\gamma})^{\frac{p+q-1}{q+\gamma}}+\frac{p-1-\gamma}{p+q-1} (|x|^\alpha\epsilon_1)^{-\frac{q+\gamma}{p+q-1}\frac{p+q-1}{p-1-\gamma}} 
=\epsilon_1  |x|^\alpha u^{p+q-1}+\frac{p-1-\gamma}{p+q-1}(|x|^\alpha\epsilon_1)^{-\frac{q+\gamma}{p-1-\gamma}},
\end{split}
\label{eq5}
 \end{align} 
 where we used Young's inequality with the exponents $\frac{p+q-1}{q+\gamma}$ and $\frac{p+q-1}{p-1-\gamma}$.  Due to $p>\gamma+1$ we have $p-\gamma-1>0$ and $p+q-1>q+\gamma$. Thus, $0<\frac{q+\gamma}{p+q-1}<1$ and $0<\frac{p-1-\gamma}{p+q-1}<1$. Using \eqref{r}, we get
 \begin{align}
 -\int_\Omega \mu(|x|)u^{p+q-1}\mathrm{d}x\leq -\mu_1\int_\Omega |x|^\alpha u^{p+q-1}\mathrm{d}x.\label{g2}
 \end{align}
 Choosing $\epsilon_1=\frac{q}{q-1}\mu_1$, then inequality \eqref{eq5} becomes
  \begin{align}
  \begin{split}
\int_\Omega u^{q+\gamma} \mathrm{d}x
\leq &\frac{q}{q-1}\mu_1 \int_\Omega |x|^\alpha u^{p+q-1} \mathrm{d}x+c_1(p,q,\mu_1,\gamma)\int_\Omega|x|^{-\alpha\frac{q+\gamma}{p-1-\gamma}} \mathrm{d}x\\
\leq &\frac{q}{q-1}\mu_1 \int_\Omega |x|^\alpha u^{p+q-1} \mathrm{d}x+c_1(p,q,\mu_1,\gamma)\int_0^R r^{1-\alpha\frac{q+\gamma}{p-1-\gamma}}\mathrm{d}r.
\end{split}
\label{eq6} 
 \end{align} 
 Substituting \eqref{g2}, \eqref{eq6} into \eqref{eq4}, we arrive at
\begin{align}
\frac{1}{q}\frac{\mathrm{d}}{\mathrm{d}t}\int_\Omega u^q\mathrm{d}x+\frac{4(q-1)}{q^2}\int_\Omega |\nabla u^{\frac{q}{2}} |^2\mathrm{d}x
\leq &\ \mu_1\int_\Omega  |x|^\alpha u^{p+q-1} \mathrm{d}x+c_1(p,q,\mu_1,\gamma)\int_0^R r^{1-\alpha\frac{q+\gamma}{p-1-\gamma}}\mathrm{d}r\nonumber\\
&+\|\kappa\|_{L^\infty(\Omega)}\int_\Omega u^q\mathrm{d}x-\mu_1\int_\Omega|x|^\alpha u^{p+q-1}\mathrm{d}x\nonumber\\
\leq &\ c \int_\Omega u^q\mathrm{d}x+c_1(p,q,\mu_1,\gamma_1)\int_0^R r^{1-\alpha\frac{q+\gamma}{p-1-\gamma}}\mathrm{d}r
\leq \ c_2 \int_\Omega u^q\mathrm{d}x+c_1.
\label{7}
\end{align}
In order to  ensure the integral $ \int_0^R r^{1-\alpha\frac{q+\gamma}{p-1-\gamma}}\mathrm{d}r$ is well defined, we need the condition $1-\alpha\frac{q+\gamma}{p-1-\gamma}>-1.$  From this, we can derive  the condition for $\alpha$, i.e. $0<\alpha<2\frac{p-1-\gamma}{q+\gamma}$. 
If we choose $\alpha=2\frac{p-1-\gamma}{q+\gamma}$ or $\alpha>2\frac{p-1-\gamma}{q+\gamma}$, then we get a contradiction for $q>1.$
Set $y(t):=\int_\Omega u^q\mathrm{d}x$. Then, \eqref{7} becomes an ODE $y'(t)\leq c_2y(t)+c_1.$ 
Finally, Lemma \ref{lemmaODE} yields $y(t)\leq \max\bigg\{y(0),\frac{c_1}{c_2}\bigg\}$. 
This implies $\|u(\cdot,t)\|_{L^q(\Omega)}\le c$ for all $t\in(0,T_{max})$, 
where the constant $c>0.$
\end{proof}
\section{$L^\infty-$bound for $u$}
In order to prove the global existence of solutions it is necessary to derive a $L^\infty$ estimate. 
Thus, we derive the $L^\infty-$bounds of $u$ and $v$ by using the $L^q-$bound from Lemma \ref{l11}.  
\begin{lem}\label{l2}
Let $\Omega$ be a bounded and smooth domain in $\mathbb{R}^n, n\geq 2.$ Let $\gamma>1$, $p>\gamma+1$, $q>\frac{n\gamma}{2}$, 
$0<\alpha<2\frac{p-1-\gamma}{q+\gamma}$ and $u$ belongs to $L^q(\Omega)$ and let $(u,v)$ be a local-in-time classical solution to \eqref{1} in $\Omega\times(0,T_{max})$ for some $T_{max}>0.$ Then, there exist $c>0$ such that 
\begin{align}
\|v\|_{W^{1,s}(\Omega)}\leq c\|u\|_{L^{q^*}(\Omega)}, \ n<s<q^*:=\frac{nq}{n\gamma-q}\ \mbox{and}\ v\in L^\infty(\Omega). \label{l1}
\end{align}

\begin{proof}
From Lemma \ref{l11} we have that  $u(\cdot,t)$ is bounded in $L^\infty((0,T_{max});L^q(\Omega)).$ Then, by the standard elliptic regularity results, cf. \cite[Theorem 19.1]{2}, applied to the second equation of \eqref{1} which warrant that  
\begin{align}
\|v\|_{W^{2,q/\gamma}(\Omega))}\leq c \|u\|_{L^{q/\gamma}(\Omega)}
\end{align}
hence $\|\nabla v\|_{W^{1,q/\gamma}(\Omega)}\leq c.$ 
The  Sobolev embedding theorem \cite[Corollary 7.11]{MR1814364}, \cite[Corollary 1.3.1]{MR2309679} or \cite{MR2424078} gives $v\in C^{m}(\overline{\Omega})$, $0\le m<2-\frac{n\gamma}{q}$ and $ \nabla v\in L^s(\Omega)$ for all $n<s<q^*:=\frac{nq}{n\gamma-q}$. In particular, we have 
$$\|v\|_{L^s(\Omega)}+\|\nabla v\|_{L^s(\Omega)}\leq c$$
for some positive constant $c>0$. Since $W^{1,s}(\Omega)\hookrightarrow L^\infty(\Omega)$, we also derive $\|v\|_{L^\infty(\Omega)}\leq c$.  \end{proof}
\end{lem}
\begin{lem}\label{l3}
Let the hypotheses of  Lemma 4.1 be satisfied and let $(u,v)$ be a local-in-time classical solution to system \eqref{1} in $\Omega\times(0,T_{max})$ for some $T_{max}>0.$ Then, there exist $c>0$ such that  
\begin{align}
\|u(\cdot,t)\|_{L^\infty((0,T_{max})\times \Omega)}\leq c \ \ \mbox{for all}\ \ t\in(0,T_{max}).
\end{align}
\end{lem}
\begin{proof}
By the variation of constants formula, we can write 
\begin{align}
u(\cdot,t)=&e^{t\Delta}u_0-\int_0^t e^{(t-\tau)\Delta}\nabla\cdot(u(\cdot,\tau)\nabla v(\cdot,\tau))\mathrm{d}\tau+
\int_0^t k(|x|)e^{(t-s)\Delta}u(\cdot,\tau)\mathrm{d}\tau-\int_0^t \mu(|x|)e^{(t-\tau)\Delta}u^p(\cdot,\tau)\mathrm{d}\tau.\label{f1}
\end{align}
Using the positivity of $u$, we can rewrite \eqref{f1} as 
\begin{align*}
u(\cdot,t)\leq&e^{t\Delta}u_0-\int_0^t e^{(t-\tau)\Delta}\nabla\cdot(u(\cdot,\tau)\nabla v(\cdot,\tau))\mathrm{d}\tau+
\int_0^t k(|x|)e^{(t-\tau)\Delta}u(\cdot,\tau)\mathrm{d}\tau.
\end{align*}
Taking sup norm on both sides of the last inequality
\begin{align}
\begin{split}
\|u(\cdot,t)\|_{L^\infty(\Omega)}\leq&\|e^{t\Delta}u_0\|_{L^\infty(\Omega)}+\int_0^t  \|e^{(t-\tau)\Delta}\nabla\cdot(u(\cdot,\tau)\nabla v(\cdot,\tau))\|_{L^\infty(\Omega)}\mathrm{d}\tau+\int_0^t \|k(|x|)e^{(t-\tau)\Delta}u(\cdot,\tau)\|_{L^\infty(\Omega)}\mathrm{d}\tau\\
=:&\|u_1(\cdot,t)\|_{L^\infty(\Omega)}+\|u_2(\cdot,t)\|_{L^\infty(\Omega)}+\|u_3(\cdot,t)\|_{L^\infty(\Omega)}.
\end{split}
\label{f2}
\end{align}
Now, we use the known smoothing estimates for the Neumann heat semigroup that is Lemma \ref{11} to estimate R.H.S of \eqref{f2}. First we estimate $\|u_1(\cdot,t)\|_{L^\infty(\Omega)}$  as follows
\begin{align}
\|u_1(\cdot,t)\|_{L^\infty(\Omega)}\leq& C_1\bigg(1+t^{-\frac{n}{2}(\frac{1}{1}-\frac{1}{\infty})}\bigg)e^{-\lambda_1t}\|u_0\|_{L^1(\Omega)}\leq C_1 (1+t^{-\frac{n}{2}})e^{-\lambda_1t}\|u_0\|_{L^1(\Omega)}.\label{f3}
\end{align} 
Next, we estimate $\|u_2(\cdot,t)\|_{L^\infty(\Omega)}:$  for all $t>0$ and $r>n$, we have 
\begin{align}
\begin{split}
\|e^{(t-\tau)\Delta}\nabla\cdot(u(\cdot,\tau)\nabla v(\cdot,\tau))\|_{L^\infty(\Omega)}
\leq& C_2\bigg(1+(t-\tau)^{-\frac{1}{2}-\frac{n}{2}(\frac{1}{r}-\frac{1}{\infty})}\bigg)e^{-\lambda_1(t-\tau)}\|u(\cdot,\tau)\nabla v(\cdot,\tau)\|_{L^r(\Omega)}\\
\leq& C_2(1+(t-\tau)^{-\frac{1}{2}-\frac{n}{2r}})e^{-\lambda_1(t-\tau)}\|u(\cdot,\tau)\nabla v(\cdot,\tau)\|_{L^r(\Omega)},
\end{split}\label{f4}
\end{align}
where $\lambda_1$ is the first non-zero eigenvalue of $-\Delta$ in $\Omega$ under the Neumann boundary condition.  Since $s>n,$ we take $n<r<s$ to estimate \eqref{f4}. By using H\"{o}lder's  inequality and Lemma \ref{l11} and Lemma \ref{l2}, we gain that 
\begin{align*}
\|u(\cdot,\tau)\nabla v(\cdot,\tau)\|_{L^r(\Omega)}
\leq& \|u(\cdot,\tau)\|_{L^{\frac{rs}{(s-r)}}(\Omega)}\|\nabla v(\cdot,\tau)\|_{L^s(\Omega)}.
\end{align*}
Next, using the interpolation inequality \eqref{ip} and \eqref{l1} in the above inequality, we obtain that 
\begin{align}
\|u(\cdot,\tau)\|_{L^{\frac{rs}{(s-r)}}(\Omega)}
\leq&  \|u(\cdot,\tau)\|_{L^\infty(\Omega)}^{\theta} \|u(\cdot,\tau)\|_{L^q(\Omega)}^{(1-\theta)},\label{f5}
\end{align}
where $\theta=1-\frac{(s-r)q}{rs}\in (0,1)$, cf. e.g. \cite[page 12]{MR4151490}. For any given $\overline{t}\in (0,T_{max})$, let us define the function $N: (0,T_{max})\rightarrow \mathbb{R}, \overline{t}\rightarrow \sup_{t\in(0,\overline{t})}\|u(\cdot,t)\|_{L^\infty(\Omega)}$ such that 
$ N(\overline{t}):=\sup_{t\in(0,\overline{t})}\|u(\cdot,t)\|_{L^\infty(\Omega)}.$
Substituting \eqref{f5} into \eqref{f4} and using Lemma \ref{l2}, we gain that
\begin{align*}
\|e^{(t-\tau)\Delta}\nabla\cdot(u(\cdot,\tau)\nabla v(\cdot,\tau))\|_{L^\infty(\Omega)}
\leq&
C_2(1+(t-\tau)^{-\frac{1}{2}-\frac{n}{2r}})e^{-\lambda_1(t-\tau)}\|u(\cdot,\tau)\|_{L^\infty(\Omega)}^{\theta} \|u(\cdot,\tau)\|_{L^q(\Omega)}^{(1-\theta)}.
\end{align*}
Therefore, we have
\begin{align*}
\|u_2(\cdot,\tau)\|_{L^\infty(\Omega)}\leq C_2\int_0^t(1+(t-\tau)^{-\frac{1}{2}-\frac{n}{2r}})e^{-\lambda_1(t-\tau)}\|u(\cdot,\tau)\|_{L^\infty(\Omega)}^{\theta} \|u(\cdot,\tau)\|_{L^q(\Omega)}^{(1-\theta)}\mathrm{d}\tau.
\end{align*}
By using Lemma \ref{l11}, the above inequality becomes
\begin{align}
\|u_2(\cdot,\tau)\|_{L^\infty(\Omega)}\leq C_2\int_0^t(1+(t-\tau)^{-\frac{1}{2}-\frac{n}{2r}})e^{-\lambda_1(t-\tau)}\|u(\cdot,\tau)\|_{L^\infty(\Omega)}^{\theta}\mathrm{d}\tau.\label{f9}
\end{align}
Finally, similar to \cite{MR3852723} -- By the order property of the Neumann heat semigroup $(e^{t\Delta})_{t\geq 0}$ due to the the maximum principle -- we estimate 
\begin{align}
\int_0^t \|k(|x|)e^{(t-\tau)\Delta}u(\cdot,\tau)\|_{L^\infty(\Omega)}\mathrm{d}\tau&\leq C_3. \label{f10}
\end{align}
Substituting \eqref{f3}, \eqref{f9} and \eqref{f10} into \eqref{f2} and recalling the definition of $N(\overline{t})$, we obtain that 
\begin{align*}
\begin{split}
 N(\overline{t})\leq C_4+C_4 N^{\theta}(\overline{t}).
\end{split}
\end{align*}
Since $\theta \in (0,1)$ and by means of \cite[Lemma 2.4]{MR4151490}, we can deduce, 
\begin{align}
N(\overline{t})\leq c,\label{final}
\end{align}
where $c$ is a positive constant. Finally, the right hand side of \eqref{final} is independent of $\overline{t}\in (0,T_{\mbox{max}})$ and thus, the uniform boundedness of $\|u(\cdot,t)\|_{L^\infty((0,T_{max});L^\infty(\Omega))}$ is obtained. 
\end{proof}
\section{Proof of Theorem 1.1}
In the last section, we prove Theorem 1.1. The proof is based on the previous results and is done by contradiction. 
\begin{proof}
The proof is based on the contradiction. Assume that $T_{max}<\infty.$ From \ref{l3} we have that 
\begin{align*}
\|u(\cdot,t)\|_{L^\infty((0,T_{max});L^\infty(\Omega))}\leq C \ \mbox{for all}\  \ t\in(0,T_{max}),
\end{align*}
Which is a contradiction to the blow-up criterion \eqref{bc}. Hence, $T_{max}=\infty$ and $\|u(\cdot,t)\|_{L^\infty((0,\infty);L^\infty(\Omega))}\leq C$. Therefore this completes the proof of main theorem. 
\end{proof}

\begin{ack}
The authors would like to thank Michael Winkler for his valuable comments on a previous version of this manuscript that eventually led to an improved presentation. A.E., supported by the Kristine Bonnevie scholarship 2020 of the Faculty of Mathematics and Natural Sciences, University of Oslo, during his research stay at Lund University in 2020, wishes to thank Erik Wahl\'en and the Centre of Mathematical Sciences, Lund University, Sweden for hosting him. A.E. was partially supported by the DFG under Germany's Excellence Strategy – MATH$^+$: The Berlin Mathematics Research Center (EXC-2046/1 – project ID: 390685689) via the project AA1-12$^*$.
\end{ack}


\bibliography{mybibfile}

\end{document}